\RequirePackage{fix-cm}

\documentclass{svjour3}                     

\usepackage{amsmath}
\usepackage{amsfonts}
\usepackage{graphicx}
\usepackage{algorithm2e}
\begin{document}

\title{A Simple Explanation for the Goldbach Conjecture}
\markright{Being Prime is not necessary}
\author{  Ameneh Farhadian$^{*}$ and Hamid Reza Fanai}


\institute{  Department of Mathematical Sciences, Sharif University of Technology, Tehran, I. R. Iran
              \email{ amenfarhad3@gmail.com, fanai@sharif.edu}           
}


\date{10/14/2022}
\maketitle

\begin{abstract}
In this paper, a simple explanation for the Goldbach Conjecture is given. We have shown that the probability of violating the conjecture not only for the prime numbers, but also for any subset of natural numbers whose distribution is similar to the prime numbers is negligible. This result makes it possible to generalize the  conjecture  to any subset of natural numbers whose distribution is similar to the prime numbers. Additionally, we selected several new subsets whose distribution amongst the natural numbers are similar to the prime numbers by randomly addition of +1 and -1 to the prime numbers and checked the Goldbach conjecture for every even integer less than  $2 \times 10^8$ by computer.  As it was expected, the Goldbach conjecture holds true for these new reconstructed sets, as well. Consequently, the conjecture can be generalized to any subset of natural numbers whose distribution is similar to the prime numbers. That is "being prime" is not necessary for the conjecture to hold for the instances. This fact brings to mind the idea that perhaps what makes the Conjecture to hold for the instances is "probability", not number theory facts.

\end{abstract}

\noindent

\section{Introduction}
The well-known Goldbach conjecture states that every even integer greater than 2 is the sum of two primes. It has been proposed in 1724 and remains unproven despite considerable effort \cite{wang2002goldbach}.  The conjecture has been shown to hold for all integers less than  $4 \times 10^{18}$ by computer \cite{e2012goldbach}. Here, we have shown that the probability of violating the conjecture not only for prime numbers, but also for any subset of natural numbers whose distribution is similar to prime numbers is negligible. Additionally,  the computational results verify that the conjecture holds true for the instances of subsets of natural numbers whose distribution are similar to the prime numbers.  Consequently, it seems that what makes the conjecture to hold true is the distribution of prime numbers in the natural numbers and  the conjecture is based on probabilistic and combinatorial facts rather than number theory.
\section{Distribution similar to prime numbers}\label{not_nec}

We know that the number of prime numbers less than or equal to some integer $n$ is denoted by $\pi(n)$. It is proved that for large $n$, $\pi(n)$ is approximated by $\frac{n}{\ln(n)}$ \cite{ireland1990classical}. Now, we define distribution similar to prime numbers.
\begin{definition}
Let $Q$ be a subset of natural numbers and $\pi_Q(n)$ be the number of elements of $Q$ less than or equal to $n$. We say  that \textit{the distribution of $Q$ in natural numbers is similar to prime numbers}, if there exists $c \in \mathbb{N}$ such that $$\vert \pi_Q(n) -\pi(n) \vert <c $$
\end{definition}

\begin{remark}
According to the definition, if $Q$ is a subset of natural numbers whose distribution is similar to prime numbers, then $\pi(n)-c<\pi_Q(n)<\pi(n)+c$. Therefore, for large $n$, which $\pi(n)$ is approximated by $\frac{n}{\ln(n)}$, we have $ \pi_Q(n)\approx \frac{n}{\ln(n)}$.
\end{remark}

By randomly addition of +1 and -1 to prime numbers, we can construct the subset $Q$ of natural numbers whose distribution is similar to prime numbers. For such subsets, we have  $\vert \pi_Q(n) -\pi(n) \vert <2$.\\
Let $P$ be the set of prime numbers and $t$ be an integer number. We define $P_t= \lbrace x+t | x \in P \rbrace $. We have $\pi_{P_t}(n)=\pi(n-t)$. Thus, $\vert \pi_{P_t}(n) -\pi(n) \vert =\vert \pi(n-t) -\pi(n) \vert<t+1 $. Thus, the distribution of $P_t$ in natural numbers is similar to prime numbers.

\begin{lemma}
Let $P$ be the set of prime numbers and $t$ be an integer number. We define $P_t= \lbrace x+t | x \in P \rbrace $. If the Goldbach conjecture holds true for the prime numbers, then the Goldbach conjecture also holds true for $P_t$ for any even integer greater than $2t+2$.
\end{lemma}
\begin{proof}
Assume the Goldbach conjecture holds true for $n>2$. We want to show that there exist $p_t, q_t \in P_t$ such that $p_t+q_t=2n$ for every even integer $2n>2t+2$.  For $2n>2t+2$, we have $2n-2t>2$. Thus, due to the Goldbach conjecture, there exist prime numbers $p$ and $q$ such that $p+q=2n-2t$. Consequently, $p+t ,q+t \in P_t$ and $ (p+t)+(q+t)=(p+q)+2t=2n$. Therefore, the Goldbach conjecture also holds true for $P_t$.
\end{proof}
The above Lemma is very simple. But, it has an important result. We see that the Goldbach conjecture holds true for a set whose elements are not prime, but with the same distribution that prime numbers have. That is being prime is not a key concept in the Goldbach conjecture. 

\begin{definition}
Let $Q$ be a subset of natural numbers. We say \textit{the Goldbach conjecture holds true for $Q$}, if there exists $N_0  \in \mathbb{N}$ such that for every even integer $2n$ greater than $N_0$, there exist $q_1, q_2 \in Q$ such that $2n=q_1+q_2$.
\end{definition}

\begin{lemma}\label{f}
Let $Q$ be a subset of natural numbers whose distribution  is similar to prime numbers. The probability that the Goldbach Conjecture does not hold for $Q$ for large even integer $2n$ is less than
\begin{equation} f(n)= \exp{(-n/\ln^2{n})}
\end{equation}
\end{lemma}

\begin{proof}
The Goldbach conjecture holds true for $2n$, if there exist $q_1, q_2 \in Q$  such that $q_1+q_2=2n$. In other words, there exist $q_1, q_2 \in Q$ such that $n$ is exactly in the middle of $q_1$ and $q_2$, i.e.  $$n-q_1=q_2-n$$
Let $A_n$ be the set of  distances of $n$  to the elements of $Q$ which are less than or equal to $n$. Let $B_n$ be the distances of $n$ to the elements of $Q$ between $n$ and $2n$, respectively. That is,
\begin{equation}
A_n := \lbrace n-q \mid \text{$ q\leq n$ and $ q \in Q$}  \rbrace
\end{equation}
\begin{equation}
B_n := \lbrace q-n \mid    \text{$ n \leq q < 2n$ and $ q \in Q$} \rbrace
\end{equation}
Clearly, we have  $A_n , B_n \subset \{ 0,1, \dots, n-1\}$.  The Goldbach conjecture does not hold true for even integer $2n$ if and only if  we have $A_n \cap B_n = \emptyset$.

 Let  $\vert A_n \vert= k_1$ and  $\vert B_n \vert= k_2$. Since  $A_n , B_n \subset \{0, 1, \dots, n-1\}$, the probability that $A_n \cap B_n = \emptyset $ is
\begin{equation}
\mathbb{P}=  \frac {       {n \choose{k_1}}         {n-k_1 \choose{k_2}}        } { {{n}\choose{k_1}}   {{n}\choose{k_2}}}
\end{equation}

This formula is similar to \cite{clarke198367} where the possible sum on primes are considered.  Here, we have defined subsets $A_n$ and $B_n$ which provides more intuition. In addition, as we will see in the section \ref{cof}, the definition of subsets $A_n$ and $B_n$ makes it possible to compute probability more exactly for the prime numbers.
We have $ \pi_Q(n) \approx \frac {n} {\ln{n}}$ for large $n$. Therefore, there are $\frac {n} {\ln{n}}$ number of $Q$ elements between 1 and $n$ and $\frac{2n} {\ln{2n}}- \frac {n}{\ln{n}} \approx \frac{n}{\ln{n}}$  $Q$ elements between $n$ and $2n$.
Substituting $k_1=k_2=\frac {n} {\ln{n}}$ in the above equation, we have
\begin{equation}
 \mathbb{P}< (\frac{n-k_1} {n})^{k_2}=(\frac{n-n/\ln{n}} {n})^{n/\ln{n}}=(1-\frac {1}{\ln {n}})^{n/\ln{n}}
\end{equation}
For large $n$, we have
\begin{equation} \mathbb{P}< \exp{(-n/\ln^2{n})}
\end{equation}  $\square$
\end{proof}

According to the above Lemma, the probability of violating the Goldbach conjecture for large enven number $2n$ not only for prime numbers but also for any subset of natural numbers whose distribution is similar to prime numbers is less than  $ f(n)=\exp{(-n/\ln^2{n})}$.

In the next section, we improve the accuracy of  the computation of the probability for the prime numbers.


\section{Improvement of the accuracy of the probability for the prime numbers}\label{cof}
In section \ref{not_nec}, we saw that for large $n$  the probability  that $A_n \cap B_n =\emptyset$  is  less than $f(n)=\exp{(-n/\ln^2{n})}$ due to the distribution of prime numbers in the natural numbers. Therefore, the probability of violating the Goldbach conjecture for any subset of natural number whose distribution is similar to prime numbers  is less than $f(n)$, too. Here, we show that "being prime" makes the function $f(n)$ to damp faster by a factor $c$, that is by function $\exp{(-cn/\ln^2{n})}$.

All prime numbers except 2 are odd. Therefore, all elements of $A_n$ (except one element) and $B_n$ are odd, if $n$ is an even number. If $n$ is odd, then all their elements except at most one of $A_n$, are even. Thus,  $A_n $ and $ B_n$ are  an even or odd subsets of $\{ 0,1, \dots, n-1\}$ with at most $\lceil n/2 \rceil$ elements. Thus, the probability that $A_n \cap B_n = \emptyset $ is
\begin{equation}
\mathbb{P}=  \frac {       { n/2 \ \choose{k_1}}         { n/2  -k_1 \choose{k_2}}        } { {{  n/2  }\choose{k_1}}   {{ n/2 }\choose{k_2}}}
\end{equation}
Due to the above computation, the probability function $f(n)$ can be  improved to $f_2(n)=\exp(-2n/\ln^2(n))$ for prime numbers. We can continue this improvement by prime number 3.
Except prime number 3, all prime numbers are congruent to 1 or 2 modulo 3. If $n$ is congruent to $i$  module 3, then all elements of $B_n$  are congruent  to $i+1$ and $i+2$ modulo 3. Thus, no one is congruent to $i$ modulo 3 in $B_n$. Thus, in the set of odd or even numbers less than $n$, there is not any element congruent to $i$  in $B_n$ and $A_n$ (except one). Thus $2/3$ of odd or even subset of $\{ 0,1, \dots,n-1\}$ are possible values for $A_n$ and $B_n$. Thus, the probability can be improved  to
\begin{equation}
\mathbb{P}=  \frac {       {n/2  \times (2/3) \choose{k_1}}         {n/2 \times (2/3)-k_1 \choose{k_2}}        } { {{n/2 \times (2/3)}\choose{k_1}}   {{n/2 \times (2/3)}\choose{k_2}}}= \frac {          {n/3-k_1 \choose{k_2}}        } {   {{n/3}\choose{k_2}}}
\end{equation}
Consequently, for large $n$ the above probability is  $\exp(-3n/\ln^2(n))$. For any prime $p$ such that $p<<n$, we can improve the probability by multiplying the coefficient in $\frac{p}{p-1}$. In other words. we have $f(n)=\exp(-cn/\ln^2(n))$  where $c=\frac{2}{1} \cdot \frac{3}{2} \cdot \frac{5}{4} \cdots \frac{p}{p-1}$. The coefficient $c$ in function $f$ makes the function to damp faster. Therefore, for the prime numbers the probability that $A_n$ and $B_n$ have no intersection  goes to zero faster than $ \exp(-n/\ln^2(n))$ as $n$ grows.\\
In this section, we have improved the accuracy of the violating probability of the conjecture considering prime numbers. We saw that being prime makes the probability function to damp faster by factor $c$. 
However, the main reason which makes the probability of  $A_n \cap B_n = \emptyset$ damps is due to the distribution of prime numbers in the natural numbers, not being prime. Because, by removing the condition of "Being Prime", still the probability function damps significantly fast. \\
In the next section, we check the damping rate of the probability function $f(n)$ by numerical calculations.

\section{Computational results}
Let's study the behavior of the function $f(n)$ introduced in the Lemma \ref{f}. The function  $f(n)=\exp{(-n/\ln^2{n})}$ is a damping function. We have $f(10000)<10^{-51}$ and $f(40000)<10^{-154}$.
As $n$ grows, the probability that the conjecture does not hold for any subset whose distribution is similar to prime numbers tends to zero.
Also, the integral of function $f$ from $N$ to infinity is negligible for large $N$. For instance, we have $\int_{x=20000}^\infty f(x) \approx 10^{-86}$  and   $\int_{x=50000}^\infty f(x) \approx 10^{-183}$.
It means that  finding a counterexample for the conjecture for a subset of natural numbers whose distribution is similar to the prime numbers  is non-probable. Therefore, the Goldbach conjecture proposed for the prime numbers can be generalized to a subset of natural numbers whose distribution is similar to the prime numbers.

\begin{conjecture}
(Generalization of Goldbach Conjecture)
Let $Q$ be a subset of natural numbers whose distribution is similar to the prime numbers. There exists $N_Q \in \mathbb{N}$, such that for any even integer $2n$ greater than $N_Q$, there exist $q_1, q_2 \in Q$ such that $q_1+q_2=2n$.
\end{conjecture}

Furthermore, we checked the above generalization in practice. We constructed several new subsets whose distribution in the natural numbers are similar to the prime numbers  by randomly addition of $+1$ and $-1$ to the prime numbers. Then, we checked the Goldbach conjecture for these sets for $2n\leq 2 \times 10^8$ by computer. As it was expected,  the Goldbach conjecture holds true for these new reconstructed sets for $2n>40$, as well.

\section{Conclusion}
In this paper, we have computed the probability that the Goldbach conjecture does not hold true for large even number for any subset of natural numbers whose distribution is similar to prime numbers. We have shown that this probability not only for prime numbers, but also for any subset of natural numbers whose distribution is similar to prime numbers is negligible.  In addition, we have shown that although being prime makes the probability function to damp faster, but the most important factor which makes the probability function to damp is the distribution, not being prime.
Consequently, the Goldbach Conjecture can be generalized to any subset of natural numbers whose distribution is similar to prime numbers.  This fact suggests the idea that perhaps the correctness of Goldbach's conjecture for the studied instances is due to "probability and combinatorics", not truths based on "number theory".


\section*{Data Availability Statement:} Data sharing is not applicable to this article as no new data other than
 given in the paper were created or analyzed in this study.
\section*{Conflict of Interest Statement:} The authors declare that there is no conflict of interest.
\bibliographystyle{plain}

%
%

\end{document}